\documentclass[11pt,letterpaper]{amsart}
\usepackage[margin=1in]{geometry}
\usepackage{amsmath,amsfonts,amsthm,amssymb}
\usepackage{enumitem,graphicx,pinlabel,thm-restate,tikz,tikz-cd,xcolor}
\usetikzlibrary{positioning}
\usepackage{stmaryrd,comment}
\usepackage[alphabetic]{amsrefs}

\newtheorem{theorem}{Theorem}[section]

\newtheorem{proposition}[theorem]{Proposition}

\newtheorem{question}[theorem]{Question} 
 
\theoremstyle{definition}
\newtheorem{definition}[theorem]{Definition}
\newtheorem{remark}[theorem]{Remark}

\usepackage{xcolor}

\definecolor{darkgreen}{rgb}{0,.5,0}

\newcommand{\ol}[1]{\overline{#1}}

\newcommand{\wt}[1]{\widetilde{#1}}

\newcommand{\N}{\mathbb{N}}
\newcommand{\Z}{\mathbb{Z}}
\newcommand{\Q}{\mathbb{Q}}

\DeclareMathOperator{\Spun}{Spun}

\begin{document}

\title{Non-concordant 2-links in simply-connected 4-manifolds}
\author{Maggie Miller}
\address{Department of Mathematics, The University of Texas at Austin, TX 78712 USA}
\email{maggie.miller.math@gmail.com}
\author{Ryan Stees}
\address{Department of Mathematics, The University of Texas at Austin, TX 78712 USA}
\email{ryan.stees@austin.utexas.edu}
\date{}


\begin{abstract}
We construct an infinite family of smooth 2-component links of 2-spheres $\Sigma_k\subset S^2\times S^2$ whose respective components are smoothly concordant but which are pairwise not topologically concordant in any topological $\Z$-homology cobordism.
In particular, the $\Sigma_k$ are pairwise non-concordant in $(S^2\times S^2)\times[0,1]$.
\end{abstract}


\maketitle

\section{Introduction} \label{sec:intro}

Two smooth (locally flat) surfaces $\Sigma_0, \Sigma_1$ in a 4-manifold are {\emph{concordant}} if there exists a smooth (resp.~locally flat) embedding $C:\Sigma\times [0,1]\hookrightarrow X\times [0,1]$ (for $\Sigma$ an abstract surface) such that $C(\Sigma\times \{0\})=\Sigma_0\times \{0\}$ and $C(\Sigma\times \{1\})=\Sigma_1\times \{1\}$. 
For smooth surfaces, locally flat concordance implies smooth concordance \cite{DaherPowell}, so there is no notion of ``exotic" concordance of surfaces analogous to exotic concordance of knots in 3-manifolds.

The study of concordance of surfaces began with work of Kervaire, who proved that all smooth 2-spheres in $S^4$ are smoothly concordant.

\begin{theorem}[\cite{kervaire}]
    Every 2-knot bounds a $B^3$ into $B^5$.
\end{theorem}

\noindent Kervaire's result contrasts the behavior of concordance one dimension lower.
The \emph{(topological) concordance group} of knots in $S^3$ is an infinitely-generated countable abelian group which exhibits rich geometric behavior \cites{Levine,CassonGordon,COT}.

More recently, Sunukjian generalized the work of Kervaire to show that concordance of connected, orientable surfaces in simply-connected 4-manifolds is determined by homology class. This statement holds in both the smooth and locally flat categories.

\begin{theorem}[\cite{sunukjian}]\label{thm:sunukjian}
If $\Sigma_0,\Sigma_1$ are genus-$g$ surfaces in a simply-connected 4-manifold $X$ and $\Sigma_0$, $\Sigma_1$ represent the same element of $H_2(X;\mathbb{Z})$, then $\Sigma_0$ and $\Sigma_1$ are concordant.
\end{theorem}

One of the motivating questions of this paper is the following well-known open problem (see e.g. \cite{cochranslice}, \cite{k3}*{Problem 4.40}).

\begin{question} \label{q:slice}
    Is every link of 2-spheres in $S^4$ slice?
\end{question}

More broadly, one may expect the behavior of link concordance in dimensions 3 and 4 to parallel the aforementioned results in knot concordance.
Even beyond the rich behavior of knot concordance in $S^3$, there are additional obstructions to link concordance, including Milnor's $\ol\mu$-invariants, which are higher-order linking numbers \cite{Milnor}.
While these invariants have been extended to surface links in $S^4$ \cite{AMY21}, they can only be nontrivial when at least one component has positive genus and therefore fail to yield counterexamples to Question~\ref{q:slice}.
This appears to provide evidence that the vast difference between the behavior of knot concordance in $S^3$ and that of 2-knot concordance in simply-connected 4-manifolds extends to link concordance and 2-link concordance.

In this paper, we show such an analogy does not hold in general.
Surprisingly, invariants analogous to Milnor's can obstruct concordance of links of spheres in other simply-connected 4-manifolds.

\begin{restatable}{theorem}{main} \label{thm:main}
There exists a nullhomologous 2-sphere $S$ smoothly embedded in $S^2\times S^2$ and nullhomologous 2-spheres $S_k\mid k\in\mathbb{Z}^{\ge 0}$ smoothly embedded in $S^2\times S^2\setminus S$ such that whenever $m\neq n$, the links $\Sigma_m=S\sqcup S_m$ and $\Sigma_n=S\sqcup S_n$ are not concordant in any topological $\Z$-homology cobordism from $S^2\times S^2$ to itself.
\end{restatable}

\noindent Note that by \cite{sunukjian}, the $S_k$ are (smoothly) concordant in $(S^2\times S^2)\times[0,1]$, so corresponding components of the $\Sigma_k$ are concordant.

In particular, there are many links of nullhomotopic 2-spheres in $S^2\times S^2$ that are not slice in $(S^2\times S^2)\times[0,1]$.
As we explain at the end of this paper, the strategy of Theorem~\ref{thm:main} can be used to construct examples of non-concordant 2-links in many other simply-connected 4-manifolds but cannot obstruct the sliceness of a 2-link in $S^4$.
Before proving Theorem~\ref{thm:main}, we provide further background on concordance of surfaces and Milnor invariants in Section~\ref{sec:background}.
Then, in Section~\ref{sec:main}, we gradually move toward the proof, beginning with a family of links in $S^3$, obstructing the concordance of a related family of knots in $S^1\times S^2$, studying knotted 2-spheres which are the spins of these knots and, finally, producing the links of 2-spheres $\Sigma_k\subset S^2\times S^2$ through surgery.

\subsection*{Notation}

We will denote the regular neighborhood of a submanifold $Y$ in a manifold $X$ by $\nu Y$ and the exterior of $Y$ in $X$ by $E_Y$.

\subsection*{Acknowledgements}

MM was partially supported by a Sloan Research Fellowship, and NSF grant DMS-2404810. Both MM and RS were partially supported by a Packard Fellowship for Science and Engineering.

\subsection*{Tool use}
No AI tools were used in this project.


\section{Background} \label{sec:background}


\subsection{Background on concordance of surfaces} \label{subsec:backgroundconc}

Let $R$ be a ring, and let $X_0$ and $X_1$ be oriented topological 4-manifolds which are (topologically) $R$-homology cobordant, that is, there exists an oriented topological 5-manifold $W$ with $\partial W=-X_0\sqcup X_1$ such that the inclusions $X_0\hookrightarrow W$ and $X_1\hookrightarrow W$ induce $R$-homology isomorphisms.
Let $\Sigma_0=\sqcup_{i=1}^m\Sigma_0^i\subset X_0$ and $\Sigma_1=\sqcup_{i=1}^m\Sigma_1^i\subset X_1$ be $m$-component locally flat, oriented surfaces with closed components.

\begin{definition} \label{def:concordance}
We say $\Sigma_0$ and $\Sigma_1$ are \emph{(topologically) $R$-homology concordant} if there exists an $m$-component oriented surface $\Sigma=\sqcup_{i=1}^m\Sigma^i$ with closed components and a locally flat, proper embedding \[C:\Sigma\times[0,1]\hookrightarrow W\] into some $R$-homology cobordism $W$ from $X_0$ to $X_1$ such that $C(\Sigma^i\times\{0\})=\Sigma_0^i$ and $C(\Sigma^i\times\{1\})=\Sigma_1^i$ for each $1\leq i\leq m$.
If $X_0=X_1$ and $W=X_0\times[0,1]$, we say $\Sigma_0$ and $\Sigma_1$ are \emph{concordant}.
\end{definition}

In the above definition, if $X_0$ and $X_1$ are instead smooth 4-manifolds which are smoothly $R$-homology cobordant, and $\Sigma_0$ and $\Sigma_1$ are smoothly embedded surfaces, then work of Daher--Powell implies $C$ may be promoted to a smooth embedding. 
In other words, concordance implies \emph{smooth concordance} for smooth surfaces in smoothly homology cobordant smooth 4-manifolds \cite{DaherPowell}.

In the setting of surfaces in a 4-manifold $X$ with $\pi_1 (X)\neq 1$, there are a few invariants that can potentially obstruct topological concordance. The {\emph{Freedman--Quinn invariant}} (\cite{freedmanquinn}, c.f.\ \cite{schneidermanteichner}) and {\emph{Stong invariant}} (\cite{stong}, in which it is called the ``Kervaire--Milnor invariant") can each obstruct concordance of knotted 2-spheres $S_0, S_1$ in $X$. Both invariants are only defined when there exists some immersed sphere $G\subset X$ intersecting $S_i$ in an odd number of points. The Freedman--Quinn invariant of $S_0, S_1$ is then valued in a quotient of the ring $\mathbb{F}_2[\{g\in \pi_1 (X)\mid g\neq 1, g^2=1\}]$. The Stong invariant automatically vanishes when the Freedman--Quinn invariant is nontrivial, but otherwise is valued in a quotient of $H_1(X;\mathbb{Z}/2\mathbb{Z})$.

Both of these invariants have somewhat restrictive use, but nevertheless lead to examples of 2-spheres in a 4-manifold that are homotopic but not concordant \cites{schneidermanteichner,stong,schwartz,klugmiller_1,klugmiller_2}. 
In \cite{millerklugstong}*{\S 5}, the first author (in joint work with M.~Klug) investigated whether the Stong invariant could prove a version of Theorem~\ref{thm:main}: do there exist componentwise homologous but non-concordant links of 2-spheres in some simply-connected 4-manifold?  They showed that the Stong invariant could, in principle, obstruct concordance from a link of null-homologous 2-spheres to the unlink. However, in order to obstruct sliceness there are extra constraints on the link. The authors were unable to find an example satisfying these conditions for which the Stong invariant is nontrivial. Conversely, Stong \cite{stong} and \cite{millerklugstong} produce examples of pairs of homologous, homologically essential 2-spheres in simply-connected 4-manifolds distinguished by the Stong invariant, but for nontrivial homology classes the discussion of \cite{millerklugstong}*{\S 5} does not apply.


\subsection{Background on Milnor invariants} \label{subsec:backgroundMilnor}

The concordance obstruction we will apply to prove Theorem~\ref{thm:main} relies on \emph{Milnor invariants} of knots in non-simply-connected, closed, orientable 3-manifolds \cites{MillerD,Heck11,Kuzbary,Stees26}.
These invariants extend Milnor's $\ol{\mu}$-invariants \cite{Milnor}, which are higher-order linking numbers for links in $S^3$.
We will take the perspective of \cite{Stees26}, which provides a convenient language for passing from knots in 3-manifolds to 2-knots in (spun) 4-manifolds, although the invariants of \cite{Kuzbary} can also be applied to obstruct concordance of the family of knots $K_k\subset S^1\times S^2$ we consider in Section~\ref{subsec:knots}.

When we use the term \emph{Milnor invariants}, we are typically refering to a broad collection of invariants which extract concordance information from knot or link groups using suitable lower central series.
Indeed, Cappell--Shaneson observed that the exterior of a concordance between links $L$ and $L'$ in the orientable 3-manifold $M$ is a homology cobordism rel boundary \cite{CappellShaneson74}.
An application of the Stallings--Dwyer Theorem \cites{Stallings,Dwyer} then implies that the quotients of the link groups $\pi:=\pi_1(E_L)$ and $\pi':=\pi_1(E_{L'})$ by their lower central subgroups, along with peripheral data, are concordance invariants.
Recall that the (classical) \emph{lower central series} of a group $\pi$ is the series of normal subgroups \[\cdots\trianglelefteq \pi_{n+1}\trianglelefteq \pi_n\trianglelefteq\cdots\trianglelefteq \pi_2\trianglelefteq \pi_1=\pi,\] where $\pi_2=[\pi,\pi]$, the commutator subgroup of $\pi$, and $\pi_{n+1}=[\pi,\pi_n]$ in general. 
We call the quotients $\pi/\pi_n$ the \emph{lower central quotients} of $\pi$. 

The above conclusion concerning the lower central quotients of link groups still holds if we replace concordance in $M\times[0,1]$ with concordance in a $\Z$-homology cobordism, where $L'$ is a link in a 3-manifold $M'$ which is $\Z$-homology cobordant to $M$.
Cappell--Shaneson observed more generally that the exterior of a concordance in a $\Z[G]$-homology cobordism is a homology cobordism rel boundary with coefficients in $\Z[G]$, so the application of the Stallings--Dwyer Theorem can be refined to a statement about the \emph{$G$-lower central series} of link groups.
In particular, we may use the $\pi_1(M)$-lower central series to study concordance in $M\times[0,1]$ \cites{Heck11,Stees26}.
We do not elaborate further on this general framework, as this paper only concerns concordance in $\Z$-homology cobordisms.

The conclusion that the lower central quotients of link groups along with peripheral data are concordance invariants is formalized in the following definition.

\begin{definition}[\cite{Stees26}*{Definition 3.1}] \label{def:basinglinks}
Let $M$ and $M'$ be closed oriented 3-manifolds.
Fix an $m$-component link $L\subset M$, and let $L'\subset M'$ be another $m$-component link.
For $n\in\N$, an \emph{$n$-basing} for $L'$ relative to $L$ is an ordered pair $(\phi,\phi_\partial)$ such that
\begin{enumerate}[label=(\arabic*)]
\item $\phi_\partial:H_1(\partial E_{L'})\xrightarrow{\cong}H_1(\partial E_L)$ is an isomorphism preserving meridians, orientations, and orderings of components, and
\item $\phi:\pi'/\pi'_n\xrightarrow{\cong}\pi/\pi_n$ is an isomorphism such that $\phi_\partial$ is \emph{compatible} with $\phi$, that is, given any basing $\tau'$ for $L'$, there exists a basing $\tau$ for $L$ such that the following diagram commutes:
\begin{center}
\begin{tikzcd}
\pi_1(\partial E_{L'}\cup\tau') \arrow[r, "\phi_{\partial*}"', "\cong"] \arrow[d] & \pi_1(\partial E_L\cup\tau) \arrow[d] \\
\pi' \arrow[d, twoheadrightarrow] & \pi \arrow[d, twoheadrightarrow] \\
\pi'/\pi'_n \arrow[r, "\cong", "\phi"'] & \pi/\pi_n.
\end{tikzcd}
\end{center}
\end{enumerate}
\end{definition}

In \cite{Stees26}, the above is called an \emph{$n$-basing over the trivial group} in order to distinguish invariants of $\Z[G]$-homology concordance for different groups $G$.
We will sometimes suppress notation and write $\phi$ instead of $(\phi,\phi_\partial)$.
Note that condition (2) above is satisfied for any basing $\tau'$ if and only if it is satisfied for a single basing.
We list some of the relevant properties of $n$-basings \cite{Stees26}:
\begin{itemize}
\item If $L'$ is $\Z$-homology concordant to $L$, an $n$-basing exists for every $n$ \cite{Stees26}*{Theorem A$'$(1)}.
\item An $n$-basing induces a $k$-basing for all $k\leq n$. 
\item An $m$-component link $L\subset S^3$ admits an $n$-basing relative to the $m$-component unlink if and only if $L$ has vanishing Milnor $\ol{\mu}$-invariants of length $n$ \cite{Stees26}*{Lemma 8.1}.
\end{itemize}

While Milnor invariants of $\Z$-homology concordance vanish for knots in $S^3$ (and, more generally, knots in $\Z$-homology 3-spheres), admitting an $n$-basing relative to a fixed knot in some other 3-manifold is a highly nontrivial condition \cites{MillerD,Kuzbary,Stees25a}.
In order to prove Theorem~\ref{thm:main}, we will begin with a family of knots $K_n\subset S^1\times S^2$ such that $K_0$ is the local unknot and for each $n>0$ the knot $K_n$ admits a $(2n+1)$-basing, but not a $(2n+2)$-basing, relative to $K_0$.
This implies the $K_n$ are pairwise non-concordant in $\Z$-homology cobordisms from $S^1\times S^2$ to itself; see Proposition~\ref{prop:knots}.
We will then consider the 2-spheres $S_n'=\Spun(K_n)\subset\Spun(S^1\times S^2)$ and obstruct the following analogous notion of $n$-basing.

\begin{definition} \label{def:basing2links}
Let $X$ and $X'$ be closed oriented topological 4-manifolds.
Fix an $m$-component link of locally flat 2-spheres $\Sigma\subset X$, and let $\Sigma'\subset X'$ be another locally flat $m$-component 2-link.
Assume $\Sigma$ and $\Sigma'$ have trivial normal bundles.
Denote the link groups of $\Sigma$ and $\Sigma'$ by $\pi$ and $\pi'$, respectively.
An \emph{$n$-basing} for $\Sigma'$ relative to $\Sigma$ is an isomorphism \[\phi:\pi'/\pi'_n\xrightarrow{\cong}\pi/\pi_n\] which preserves meridians, that is, given any basing $\tau'$ for $\Sigma'$, there exists a basing $\tau$ for $\Sigma$ such that $\phi$ preserves the ordered based meridians induced by $\tau'$ and $\tau$.
\end{definition}

In future work, we will consider an extension of Definition~\ref{def:basing2links} to more general surface links.

\begin{proposition}
A topological concordance between $\Sigma$ and $\Sigma'$ in a topological $\Z$-homology cobordism between $X$ and $X'$ induces an $n$-basing for $\Sigma'$ relative to $\Sigma$ for each $n$.
\end{proposition}

See \cite{Stees26}*{Corollary 4.3} for a similar statement for links in 3-manifolds.

\begin{proof}
Suppose the $m$-component links $\Sigma\subset X$ and $\Sigma'\subset X'$ are concordant in the $\Z$-homology cobordism $W$ via the concordance $C$.
Take the basepoint of $X$ to also be the basepoint of $W$, and choose a path $\alpha$ from the basepoint of $X$ to the basepoint of $X'$.
Assume the basepoints of $X$ and $X'$ are outside $\nu C$, and assume $\alpha$ misses $\nu C$.
Let $E_C^\times$ denote the reduced space $E_C/\alpha$, which is homotopy equivalent to $E_C$ and conveniently identifies the basepoints of $X$ and $X'$.
Let $\pi$, $\pi'$, and $\pi^C$ denote the fundamental groups of $E_\Sigma$, $E_{\Sigma'}$, and $E_C^\times$, respectively.

As $E_C$ is a homology cobordism rel boundary, the inclusions $E_\Sigma\hookrightarrow E_C$ and $E_{\Sigma'}\hookrightarrow E_C$ induce isomorphisms on $H_1(-)$ and $H_2(-)$.
By Hopf's Theorem for path-connected spaces (see \cite{Brown}*{Section II.5} and \cite{EilenbergMacLane}), for any path-connected space $A$, $H_1(A)\to H_1(\pi_1(A))$ is an isomorphism and $H_2(A)\to H_2(\pi_1(A))$ is a surjection, so these inclusions induce isomorphisms \[H_1(\pi)\xrightarrow{\cong} H_1(\pi^C)\xleftarrow{\cong} H_1(\pi')\] and surjections \[H_2(\pi)\twoheadrightarrow H_2(\pi^C)\twoheadleftarrow H_2(\pi').\]
By the Stallings--Dwyer Theorem \cites{Stallings,Dwyer}, we obtain isomorphisms for all $n$ \[\pi/\pi_n\xrightarrow{\cong}\pi^C/\pi^C_n\xleftarrow{\cong}\pi'/\pi'_n.\]

Now fix $n$, and let $\tau_0$ and $\tau'$ be any basings in $X$ and $X'$ for $\Sigma$ and $\Sigma'$, respectively.
Choose paths $\beta$ along the concordance $C$ from the endpoints of $\tau_0$ to the corresponding endpoints of $\tau'$.
After pushing off into $E_C^\times$, the resulting wedge of circles $\gamma=\tau_0\beta(\tau')^{-1}$ has factors representing elements of $\pi^C$.
Using the composite $\pi^C\twoheadrightarrow\pi^C/\pi^C_n\xrightarrow{\cong}\pi/\pi_n$, map each factor of $\gamma$ to an element of $\pi/\pi_n$, and lift that element to some element of $\pi$.
The resulting collection of elements in $\pi$ can be represented by a wedge of circles $\gamma_0$ in $E_\Sigma$.
The basing $\tau=\gamma_0^{-1}\tau_0$ for $\Sigma$ satisfies Definition~\ref{def:basing2links} for the basing $\tau'$.
(Note that in $E_C^\times$, we have a homotopy of paths rel endpoints $\beta^{-1}\mu\beta\simeq\mu'$, where $\mu$ and $\mu'$ are the unbased meridians of $\Sigma$ and $\Sigma'$, respectively, which drags the meridians $\mu$ along $\beta$.)
Thus, the isomorphism $\pi'/\pi'_n\xrightarrow{\cong}\pi/\pi_n$ induced by the concordance is an $n$-basing for $\Sigma'$ relative to $\Sigma$.
\end{proof}


\section{Non-concordant links of 2-spheres in simply-connected 4-manifolds} \label{sec:main}

In this section, we build toward the proof of Theorem~\ref{thm:main}.
We start with a family of 2-component links $L_k\subset S^3$ from \cite{Milnor}; see Figure \ref{fig:milnorlink}.
Performing 0-surgery on the first component of each $L_k$, we obtain a family of knots $K_k\subset S^1\times S^2$ which we show are pairwise non-concordant in $\Z$-homology cobordisms in Section~\ref{subsec:knots}; see Proposition~\ref{prop:knots}.
We use this result to show that the spun 2-knots \[S_k':=\Spun(K_k)\subset\Spun(S^1\times S^2)\] are pairwise non-concordant in $\Z$-homology cobordisms in Section~\ref{subsec:spheres}; see Proposition~\ref{prop:spheres}.
Finally, in Section~\ref{subsec:links}, we surger $\Spun(S^1\times S^2)$ along an $S^1$ to obtain $S^2\times S^2$.
Consider the family of 2-component links of nullhomotopic 2-spheres \[\Sigma_k=S\sqcup S_k\subset S^2\times S^2,\] where $S$ is the meridional 2-sphere to the surgery curve and $S_k$ is the image of $S_k'$ under the surgery to $S^2\times S^2$.
By \cite{sunukjian}, corresponding components of the $\Sigma_k$ are concordant in $(S^2\times S^2)\times [0,1]$.
We leverage the non-concordance of the spheres $S_k'$ to show the links $\Sigma_k$ are pairwise non-concordant (in $\Z$-homology cobordisms).


\subsection{Non-concordant knots in $S^1\times S^2$} \label{subsec:knots}

We begin with a family of 2-component links $L_k=L_k^1\sqcup L_k^2\subset S^3$, $k\geq 0$, from \cite{Milnor}, where $L_0$ is the 2-component unlink.
The link $L_1$, which is the (positive-clasped) Whitehead link, and the general link $L_k$ are pictured in Figure \ref{fig:milnorlink}.
These links are pairwise not concordant in $\Z$-homology cobordisms, distinguished by Milnor's $\ol{\mu}$-invariants \cites{Milnor,Cochran90}.
Milnor's $\ol{\mu}$-invariants of $L_k$ of length $2k+1$ vanish, but the length $2k+2$ invariant $\ol{\mu}(1\cdots122)$ of $L_k$ is $(-1)^k$.
In the language of Section~\ref{subsec:backgroundMilnor}, $L_k$ admits a $(2k+1)$-basing relative to the 2-component unlink in $S^3$ but does not admit a $(2k+2)$-basing.
By properties of $k$-basings \cite{Stees26}, for $m>n$ the link $L_m$ admits a $(2n+1)$-basing but not a $(2n+2)$-basing relative to $L_n$.

\begin{figure}
    \labellist
    \pinlabel{$k$} at 191 85
      \pinlabel{$L_1^1$} at 5 60
      \pinlabel{$L_1^2$} at 15 5
      \pinlabel{$L_k^1$} at 90 60
      \pinlabel{$L_k^2$} at 100 5
    \pinlabel{\scalebox{2}{$\cdots$}} at 205 35
    \endlabellist
    \vspace{.1in}\includegraphics[width=120mm]{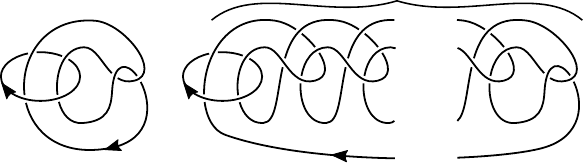}\caption{Left: $L_1$ is the positive-clasped Whitehead link. Right: The general link $L_k$.}\label{fig:milnorlink}
\end{figure}
We now pass from the family $L_k$ to a family of pairwise non-concordant knots $K_k\subset S^1\times S^2$.
Let $K_k\subset S^1\times S^2$ be obtained from $L_k\subset S^3$ as the image of $L_k^2$ under 0-surgery on $S^3$ along $L_k^1$.
The knot $K_0$ is the local unknot.

\begin{proposition} \label{prop:knots}
For $m>n$, the knot $K_m$ admits a $(2n+1)$-basing but does not admit a $(2n+2)$-basing relative to $K_n$.
Consequently, the knots $K_k$ are pairwise not $\Z$-homology concordant.
\end{proposition}

\begin{proof}
Fix $n\geq 1$.
We first show $K_n$ admits a $(2n+1)$-basing, but not a $(2n+2)$-basing, relative to $K_0$, the local unknot.
Let $\pi=\pi_1(E_{K_n})$, and denote the 2-generator free group $\pi_1(E_{K_0})$ by $F$.

Let $(\phi_{2n+1},\phi_\partial)$ be any $(2n+1)$-basing for $L_n$ relative to $L_0$, which exists because Milnor's $\ol{\mu}$-invariants of $L_n$ of lengths $2n+1$ vanish \cite{Stees26}*{Lemma 8.1}.
The knot groups $F$ and $\pi$ are related to the corresponding link groups by $F=\pi_1(E_{L_0})$ and $\pi=\pi_1(E_{L_n})/\langle\langle \ell_1\rangle\rangle$, where $\ell_1$ denotes the (based) longitude of $L_n^1$.
Note that the boundary isomorphism $\phi_\partial$ is determined by the requirement that the square in Definition~\ref{def:basinglinks} commutes, i.e.~it must preserve 0-framed longitudes.
Thus, as $\phi_{2n+1}$ is a $(2n+1)$-basing, $\phi_{2n+1}(\ell_1)=1$, so $\phi_{2n+1}$ descends to an isomorphism \[\ol{\phi}_{2n+1}:\pi/\pi_{2n+1}\xrightarrow{\cong}F/F_{2n+1}.\]

Let $\ol\phi_\partial$ denote the induced boundary isomorphism, that is, the restriction of $\phi_\partial$ to the second component.
Any basing $\tau_n$ for $K_n$ lifts to a basing $\wt\tau_n$ for $L_n$ because the inclusion-induced homomorphism  $\pi_1(E_{L_n})\twoheadrightarrow\pi$ is surjective.
Let $\wt\tau_0$ be a basing for $L_0$ corresponding to $\wt\tau_n$ as in Definition~\ref{def:basinglinks}, guaranteed by the existence of the $(2n+1)$-basing for $L_n$ relative to $L_0$, and consider the induced basing $\tau_0$ for $K_0$.
By construction, $\tau_n$, $\tau_0$, $
\ol\phi_{2n+1}$, and $\ol\phi_\partial$ make the square in Definition~\ref{def:basinglinks} commute, so $(\ol\phi_{2n+1},\ol\phi_\partial)$ is a $(2n+1)$-basing for $K_n$ relative to $K_0$.

Now suppose, for the sake of contradiction, that there exists a $(2n+2)$-basing $(\ol{\phi}_{2n+2},\ol\phi_\partial)$ for $K_n$ relative to $K_0$. 
Choose any basing $\tau_n$ for $K_n$ and let $\tau_0$ be a basing for $K_0$ compatible with $(\ol{\phi}_{2n+2},\ol\phi_\partial)$ as in Definition~\ref{def:basinglinks}.
Lift $\tau_n$ to a basing $\wt\tau_n$ for $L_n$.
As $\pi_1(E_{L_0})=F$, choose the basing $\wt\tau_0$ for $L_0$ using the basing $\tau_0$, along with a basing $\wt\tau_0^1$ for $L_0^1$ such that the image of the corresponding based meridian in $F/F_{2n+2}$ agrees with the image of the meridian corresponding to $\wt\tau_n^1$ under \[\pi_1(E_{L_n})\twoheadrightarrow\pi\twoheadrightarrow\pi/\pi_{2n+2}\xrightarrow{\cong}F/F_{2n+2}.\]
Let $\lambda\in\pi$ denote the based class of the 0-framed longitude of $K_n$ (which is well-defined since $K_n$ is nullhomologous), and let $\ell_1$ and $\ell_2$ be the 0-framed longitudes for $L_n^1$ and $L_n^2$, respectively, determined by $\wt\tau_n$.

Recall the Chen-Milnor presentation \cite{Milnor} \[\pi_1(E_{L_n})/\pi(E_{L_n})_{2n+2}=\langle\, m_1, m_2\,|\,F_{2n+2}, [m_1,\ell_1],[m_2,\ell_2]\,\rangle,\] where $m_1$ and $m_2$ are the meridians determined by $\wt\tau_n$.
This yields the presentation \[\pi/\pi_{2n+2}=\langle\, m_1, m_2\,|\,F_{2n+2}, \ell_1,[m_2,\lambda]\,\rangle.\]
We thus obtain a surjection from $F/F_{2n+2}$ to itself via the composition \[F/F_{2n+2}\twoheadrightarrow\pi_1(E_{L_n})/\pi(E_{L_n})_{2n+2}\twoheadrightarrow\pi/\pi_{2n+2}\xrightarrow[\cong]{\ol\phi_{2n+2}}F/F_{2n+2}.\]
As $F/F_{2n+2}$ is Hopfian, this surjection is an isomorphism, and so \[\pi_1(E_{L_n})/\pi(E_{L_n})_{2n+2}\to F/F_{2n+2}=\pi_1(E_{L_0})/\pi_1(E_{L_0})_{2n+2}\] is also an isomorphism.
This isomorphism sends $\ell_1\mapsto 1$, as $\ell_1$ becomes trivial in the quotient $\pi/\pi_{2n+2}$, and also sends $\ell_2\mapsto\lambda\mapsto 1$.
Thus, it constitutes a $(2n+2)$-basing for $L_n$ relative to $L_0$.
This contradicts the nontriviality of the length $(2n+2)$ invariant $\ol\mu(1\cdots 122)$.

Finally, given $m>n$, $K_m$ and $K_n$ admit $(2m+1)$- and $(2n+1)$-basings, respectively, relative to $K_0$.
As $m>n$, the $(2m+1)$-basing for $K_m$ relative to $K_0$ induces a $(2n+2)$-basing and a $(2n+1)$-basing which, composed with the inverse of the $(2n+1)$-basing for $K_n$ relative to $K_0$, yields the desired $(2n+1)$-basing for $K_m$ relative to $K_n$ (see \cite{Stees26}*{Section 4.1}).
If a $(2n+2)$-basing for $K_m$ relative to $K_n$ existed, we could post-compose its inverse with the $(2n+2)$-basing for $K_m$ relative to $K_0$ to obtain a $(2n+2)$-basing for $K_n$ relative to $K_0$, a contradiction.
Thus, the knots $K_k$ are pairwise not $\Z$-homology concordant.
\end{proof}

In the language of \cite{Stees26}*{Theorem D}, the argument above shows that for $m>n$ the Milnor invariant $\ol\mu_{2n+1}(K_m)$ over the trivial group relative to $K_n$ is nonvanishing.
One could also use Kuzbary's \emph{Dwyer number} \cite{Kuzbary} to conclude the $K_k$ are pairwise non-concordant.
We used the argument above as it more closely relates to the proof of Proposition~\ref{prop:spheres} to come.


\subsection{Non-concordant 2-spheres through spinning}\label{subsec:spheres}

We first recall the notion of ``spinning," which yields a 2-sphere in a smooth 4-manifold given the input of a knot in a closed 3-manifold.

\begin{definition}
    Let $K$ be a knot in a 3-manifold $Y$. Let $B$ be a ball centered about a point in $K$, intersecting $K$ in a trivial 1-stranded tangle.
    
    The {\emph{spin}} $(\Spun(Y),\Spun(K))$ of $(Y,K)$ is the pair \[((Y\setminus \mathring{B}, K\setminus\mathring{B})\times S^1) \cup (S^2\times D^2,S^0\times D^2),\] where the gluing map is the canonical identification of each boundary with $S^2\times S^1$.
\end{definition}

\begin{remark}
    A simple application of the Seifert--van Kampen theorem yields $\pi_1(Y)=\pi_1(\Spun(Y))$ and $\pi_1(Y\setminus K)=\pi_1(\Spun(Y)\setminus \Spun(K))$, with isomorphisms induced by the canonical inclusion (after puncturing). 
\end{remark}

Let $S_k'$ be the family of smooth 2-spheres $S_k'=\Spun(K_k)$ in the smooth 4-manifold $\Spun(S^1\times S^2)$.

\begin{proposition} \label{prop:spheres}
For $m>n$, the 2-knot $S_m'$ admits a $(2n+1)$-basing relative to $S_n'$, but does not admit a $(2n+3)$-basing.
Consequently, the 2-knots $S_k'$ are pairwise not $\Z$-homology concordant.
\end{proposition}

Before presenting the proof, we remark on subtle differences between the settings of Propositions ~\ref{prop:knots} and~\ref{prop:spheres}.
\begin{itemize}
\item A $k$-basing of 2-knots requires less structure than a $k$-basing of knots, as the peripheral structures no longer contain longitudes.
Thus, obstructing a $k$-basing of knots (as  in Proposition~\ref{prop:knots}) does not immediately obstruct a $k$-basing of the corresponding spun 2-knots (as in Proposition~\ref{prop:spheres}).
This explains why we obstruct a $(2n+2)$-basing in Propsosition~\ref{prop:knots}, but only obstruct a $(2n+3)$-basing in Proposition~\ref{prop:spheres}. 
\item While 2-knots do not have longitudes geometrically, spun 2-knot groups have elements---namely, the images of the longitudes of the knots being spun---which algebraically behave like longitudes.
Our proof below leverages the role these elements play in the lower central quotients of the 2-knot groups.
\end{itemize}

\begin{proof}[Proof of Proposition~\ref{prop:spheres}]
We use hands-on group-theoretic arguments.
Fix $n\geq 1$.
As in the proof of Proposition~\ref{prop:knots}, we first show $S_n'$ admits a $(2n+1)$-basing, but not a $(2n+3)$-basing, relative to $S_0'$.
As before, let $\pi=\pi_1(E_{S_n'})=\pi_1(E_{K_n})$, and let $F$ be the 2-generator free group $F=\pi_1(E_{S_0'})=\pi_1(E_{K_0})$.
The existence of a $(2n+1)$-basing for $S_n'$ relative to $S_0'$ is immediate from Proposition~\ref{prop:knots} and the isomorphisms between knot and spun 2-knot groups.
The $(2n+1)$-basing for $K_n$ relative to $K_0$ induces a $(2n+1)$-basing for $S_n'$ relative to $S_0'$.

Now suppose, for the sake of contradiction, that there exists a $(2n+3)$-basing \[\phi:\pi/\pi_{2n+3}\xrightarrow{\cong}F/F_{2n+3}\] for $S_n'$ relative to $S_0'$.
We show this implies the existence of a $(2n+2)$-basing for $K_n$ relative to $K_0$, contradicting Proposition~\ref{prop:knots}.
Choose basings $\tau_n$ and $\tau_0$ for $S_n'$ and $S_0'$ compatible with $\phi$ in the sense of Definition~\ref{def:basing2links}. Let $\mu$ and $\mu_0$ denote the based meridians of $S_n'$ and $S_0'$, respectively.
Because of the canonical isomorphism $\pi_1(E_{S_n'})=\pi_1(E_{K_n})$, $\mu$ is represented by a based meridian for $K_n\subset S^1\times S^2$.
Let $\lambda\in\pi$ be the corresponding based 0-framed longitude, viewed as an element of $\pi=\pi_1(E_{S_n'})$.

The relation $[\mu,\lambda]=1$ holds in $\pi$, so it also holds in $\pi/\pi_{2n+3}$.
Therefore, \[[\phi(\mu),\phi(\lambda)]=[\mu_0,\phi(\lambda)]=1\in F/F_{2n+3}.\]
An application of the Magnus expansion \cite{MKS} implies $\phi(\lambda)\in F_{2n+2}$, so $\lambda\in\pi_{2n+2}$ since $\phi$ is an isomorphism.
Thus, viewed as an isomorphism on lower central quotients of the knot groups $K_n$ and $K_0$, $\phi$ induces a $(2n+2)$-basing for $K_n$ relative to $K_0$ \[\ol\phi:\pi/\pi_{2n+2}\xrightarrow{\cong}F/F_{2n+2}\] (where the basings $\tau_n$ and $\tau_0$ correspond), which contradicts Proposition~\ref{prop:knots}.

The remainder of the argument is entirely similar to the proof of Proposition~\ref{prop:knots}.
Given $m>n$, $S_m'$ and $S_n'$ admit $(2n+1)$-basings relative to $S_0'$, hence $S_m'$ admits a $(2n+1)$-basing relative to $S_n'$.
If $S_m'$ admitted a $(2n+3)$-basing relative to $S_n'$, one could obtain a $(2n+3)$-basing for $S_n'$ relative to $S_0'$, contradicting the above argument.
Thus, the $S_k'$ are pairwise not $\Z$-homology concordant.
\end{proof}

The spin of $(Y,K)$ can also be defined as \[\Spun(Y,K)=(\partial ((Y\setminus \mathring{B})\times D^2), \partial((K\setminus\mathring{B})\times D^2)),\] where $B$ is a small ball in $Y$ intersecting $K$ in a trivial 1-stranded tangle. This definition makes the extension to \emph{super-spinning} \cite{Cappell} more obvious: the $m$-super-spin of $(Y,K)$ is \[\Spun_m(Y,K)=(\Spun_m(Y),\Spun_m(K))=(\partial ((Y\setminus \mathring{B})\times D^{m+1}), \partial((K\setminus\mathring{B})\times D^{m+1})).\]
We again have canonical isomorphisms \[\pi_1(Y)\to\pi_1(\Spun_m(Y)) \hspace{1em}\text{ and }\hspace{1em} \pi_1(Y\setminus K)\to\pi_1(\Spun_m(Y)\setminus\Spun_m(K)),\] so 
Proposition~\ref{prop:spheres} also holds for $\Spun_m(S^1\times S^2, K_k)$. 
An analogue of Theorem~\ref{thm:main} follows in a similar manner.
Thus, for each $n\geq 2$, this construction yields infinite families of non-concordant 2-component links of $n$-spheres in specific simply-connected $(n+2)$-manifolds.


\subsection{Proof of Theorem~\ref{thm:main}} \label{subsec:links}

\main*

\begin{proof}
Consider the smooth 4-manifold \[X:=\left(\Spun(S^1\times S^2)-(S^1\times B^3)\right)\cup(D^2\times S^2), \] obtained from $\Spun(S^1\times S^2)$ by performing surgery on \[S^1\times\{\ast\}\times\{\ast\}\subset ((S^1\times S^2)\setminus\mathring{B}^3)\times S^1\subset\Spun(S^1\times S^2).\]
The coordinates on $(S^1\times S^2)\times S^1$ induce a parametrization $\nu(S^1)=S^1\times B^3$ and hence $\partial\nu(S^1)=S^1\times S^2$. We use the natural identification of $\partial D^2\times S^2=S^1\times S^2$ to perform the surgery.
Denote the meridional $S^2$ to this surgery curve by $S$.
A Kirby calculus argument, seen in Figure \ref{fig:kirby}, shows that $X\cong S^2\times S^2$.
The two $S^1$ surgeries which take $(S^1\times S^2)\times S^1$ to $X$ correspond to replacing the 1-handles represented by the two dotted circles in Figure \ref{fig:kirby} (left) with 0-framed 2-handles.
We obtain the standard diagram for $S^2\times S^2$ by handle slides.

\begin{figure}
\labellist
    \pinlabel{$(S^1\times S^2)\times S^1$} at 42 105
    
    \pinlabel{\textcolor{red}{$0$}} at 52 90
\pinlabel{\textcolor{red}{$0$}} at 75 90

\pinlabel{$\Spun(S^1\times S^2)$} at 135 105

\pinlabel{\textcolor{red}{$0$}} at 147 90
\pinlabel{\textcolor{red}{$0$}} at 170 90
\pinlabel{\textcolor{red}{$0$}} at 135 5

\pinlabel{$X$} at 232 105
\pinlabel{$X$} at 327 105
\pinlabel{$X$} at 427 105

\pinlabel{\textcolor{red}{$0$}} at 243 90
\pinlabel{\textcolor{red}{$0$}} at 266 90
\pinlabel{\textcolor{red}{$0$}} at 231 5
\pinlabel{\textcolor{red}{$0$}} at 215 90

\pinlabel{\textcolor{red}{$0$}} at 338 95
\pinlabel{\textcolor{red}{$0$}} at 327 5
\pinlabel{\textcolor{red}{$0$}} at 310 90

\pinlabel{\textcolor{red}{$0$}} at 435 90
\pinlabel{\textcolor{red}{$0$}} at 455 90
\pinlabel{\textcolor{red}{$0$}} at 445 8
\pinlabel{\textcolor{red}{$0$}} at 400 70
    \endlabellist

    \vspace{.2in}
    \includegraphics[width=150mm]{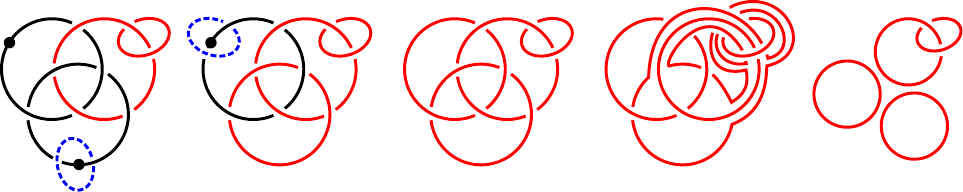}
    \caption{First: The 4-manifold $(S^1\times S^2)\times S^1$. We highlight the curve $\{\ast\}\times\{\ast\}\times S^1$. 
    Second: The 4-manifold $\Spun(S^1\times S^2)$, obtained from $(S^1\times S^2)\times S^1$ by surgery along $\{\ast\}\times\{\ast\}\times S^1$. We highlight the curve $S^1\times\{\ast\}\times\{\ast\}$.
    Third: The 4-manifold $X$ obtained by surgering $\Spun(S^1\times S^2)$ along $S^1\times\{\ast\}\times\{\ast\}$.
    Fourth: We perform two handle slides to the diagram for $X$.
    Fifth: Isotopy of the diagram shows $X\cong S^2\times S^2$.}\label{fig:kirby}
\end{figure}

Let $\{\Sigma_k\}$ be the family of smoothly embedded links $\Sigma_k=S\sqcup S_k\subset S^2\times S^2$, where $S_k$ is the image of $S_k'\subset \Spun(S^1\times S^2)$ under the surgery to $S^2\times S^2$.
Suppose that $\Sigma_m$ and $\Sigma_n$ are topologically $\Z$-homology concordant for $m\neq n$.
Performing surgery on the concordance from $S\subset \Sigma_m$ to $S\subset \Sigma_n$, replacing an $(S^2\times D^2)\times[0,1]$ with $(B^3\times S^1)\times [0,1]$, yields a concordance from $S_m'$ to $S_n'$ in a $\Z$-homology cobordism from $\Spun(S^1\times S^2)$ to itself, contradicting Proposition~\ref{prop:spheres}.
Thus, the links $\Sigma_k$ are pairwise not $\Z$-homology concordant.
\end{proof}

The construction of the family of links $\Sigma_k\subset S^2\times S^2$ provides a general strategy for producing non-concordant links of 2-spheres in various simply-connected 4-manifolds; however, it cannot possibly provide a positive answer to Question \ref{q:slice} regarding 2-links in $S^4$.
In the construction of the $\Sigma_k$, the 4-manifold $S^2\times S^2$ is the result of surgering out the fundamental group of $\Spun(S^1\times S^2)$.
This procedure works for more general 3-manifolds $Y$ which are not rational homology spheres.
For $\Q$-homology spheres $Y$, $H_2(Y)=0$ forces $H_2(\pi)=0$, where $\pi$ is the knot group of any knot in $Y$.
For any pair of homotopic knots in $Y$, we thus obtain meridian-preserving isomorphisms between the quotients of the knot groups by their lower central subgroups by the Stallings exact sequence \cite{Stallings}.
Consequently, the specific type of Milnor invariant we apply to prove Theorem~\ref{thm:main} is unable to obstruct the concordance of the spins of homotopic knots in $Y$.
Since $Y$ must have nontrivial $H_2$, so must its spin, hence any 4-manifold to which our proof applies fundamentally must have nontrivial $H_2$.

Spinning knots in 3-manifolds $Y$ can provide further examples of non-concordant 2-knots in non-simply-connected 4-manifolds using other versions of Milnor invariants.
Indeed, one may apply arguments similar to those in Section~\ref{subsec:spheres} to show that the family of non-concordant knots $K^n_\alpha$ from \cite{Stees25a}*{Theorem 5.1}, which are constructed from any fixed $\pi_1$-essential knot in any aspherical 3-manifold, yield non-concordant 2-spheres after spinning.


\bibliography{references.bib}

@book {Brown,
    AUTHOR = {Brown, Kenneth S.},
     TITLE = {Cohomology of groups},
    SERIES = {Graduate Texts in Mathematics},
    VOLUME = {87},
 PUBLISHER = {Springer-Verlag, New York-Berlin},
      YEAR = {1982},
     PAGES = {x+306},
      ISBN = {0-387-90688-6},
   MRCLASS = {20-02 (18-01 20F32 20J05 55-01)},
  MRNUMBER = {672956},
MRREVIEWER = {Ross\ Staffeldt},
}

@article {CappellShaneson74,
    AUTHOR = {Cappell, Sylvain E. and Shaneson, Julius L.},
     TITLE = {The codimension two placement problem and homology equivalent
              manifolds},
   JOURNAL = {Ann. of Math. (2)},
  FJOURNAL = {Annals of Mathematics. Second Series},
    VOLUME = {99},
      YEAR = {1974},
     PAGES = {277--348},
      ISSN = {0003-486X},
   MRCLASS = {57D40},
  MRNUMBER = {339216},
MRREVIEWER = {J.\ P.\ Levine},
       DOI = {10.2307/1970901},
       URL = {https://doi.org/10.2307/1970901},
}

@article {Cochran90,
    AUTHOR = {Cochran, Tim D.},
     TITLE = {Derivatives of links: {M}ilnor's concordance invariants and
              {M}assey's products},
   JOURNAL = {Mem. Amer. Math. Soc.},
  FJOURNAL = {Memoirs of the American Mathematical Society},
    VOLUME = {84},
      YEAR = {1990},
    NUMBER = {427},
     PAGES = {x+73},
      ISSN = {0065-9266,1947-6221},
   MRCLASS = {57M25 (55S30)},
  MRNUMBER = {1042041},
MRREVIEWER = {C.\ Kearton},
       DOI = {10.1090/memo/0427},
       URL = {https://doi.org/10.1090/memo/0427},
}

@book {K3,
     TITLE = {K3---a new problem list in low-dimensional topology},
    SERIES = {Mathematical Surveys and Monographs},
    VOLUME = {295},
    EDITOR = {Baykur, R. \.Inan{\c c} and Kirby, Robion C. and Ruberman,
              Daniel},
 PUBLISHER = {American Mathematical Society, Providence, RI},
      YEAR = {2026},
     PAGES = {xi+430},
      ISBN = {978-1-4704-8433-0},
   MRCLASS = {57-02 (57Kxx 57Mxx)},
  MRNUMBER = {5077826},
       DOI = {10.1090/surv/295},
       URL = {https://doi.org/10.1090/surv/295},
}

@book {cappell,
    AUTHOR = {Cappell, Sylvain Edward},
     TITLE = {S{UPER}-{SPINNING} {AND} {KNOT} {COMPLEMENTS}},
      NOTE = {Thesis (Ph.D.)--Princeton University},
 PUBLISHER = {ProQuest LLC, Ann Arbor, MI},
      YEAR = {1969},
     PAGES = {49},
   MRCLASS = {Thesis},
  MRNUMBER = {2619165},
       URL =
              {http://gateway.proquest.com/openurl?url_ver=Z39.88-2004&rft_val_fmt=info:ofi/fmt:kev:mtx:dissertation&res_dat=xri:pqdiss&rft_dat=xri:pqdiss:7014192},
}

@misc{millerklugstong,
      title={Concordance of spheres in 4-manifolds with an immersed dual sphere}, 
      author={Michael Klug and Maggie Miller},
      year={2022},
      eprint={2211.07177},
      archivePrefix={arXiv},
      primaryClass={math.GT},
      url={https://arxiv.org/abs/2211.07177}, 
}

@article{AMY21,
  author={Audoux, Benjamin and Meilhan, Jean-Baptiste and Yasuhara, Akira},
  year={2021},
  eprint={https://arxiv.org/abs/2109.14578},
  status={preprint},
  title={Milnor-type invariants for surface-links and cut-diagrams},
}

@incollection {CassonGordon,
    AUTHOR = {Casson, A. J. and Gordon, C. McA.},
     TITLE = {On slice knots in dimension three},
 BOOKTITLE = {Algebraic and geometric topology ({P}roc. {S}ympos. {P}ure
              {M}ath., {S}tanford {U}niv., {S}tanford, {C}alif., 1976),
              {P}art 2},
    SERIES = {Proc. Sympos. Pure Math.},
    VOLUME = {XXXII},
     PAGES = {39--53},
 PUBLISHER = {Amer. Math. Soc., Providence, RI},
      YEAR = {1978},
      ISBN = {0-8218-1433-8},
   MRCLASS = {57M25},
  MRNUMBER = {520521},
}

@article {cochranslice,
    AUTHOR = {Cochran, Tim},
     TITLE = {Slice links in {$S\sp{4}$}},
   JOURNAL = {Trans. Amer. Math. Soc.},
  FJOURNAL = {Transactions of the American Mathematical Society},
    VOLUME = {285},
      YEAR = {1984},
    NUMBER = {1},
     PAGES = {389--401},
      ISSN = {0002-9947,1088-6850},
   MRCLASS = {57Q45},
  MRNUMBER = {748845},
MRREVIEWER = {Jonathan\ A.\ Hillman},
       DOI = {10.2307/1999487},
       URL = {https://doi.org/10.2307/1999487},
}

@article {COT,
    AUTHOR = {Cochran, Tim D. and Orr, Kent E. and Teichner, Peter},
     TITLE = {Knot concordance, {W}hitney towers and {$L^2$}-signatures},
   JOURNAL = {Ann. of Math. (2)},
  FJOURNAL = {Annals of Mathematics. Second Series},
    VOLUME = {157},
      YEAR = {2003},
    NUMBER = {2},
     PAGES = {433--519},
      ISSN = {0003-486X,1939-8980},
   MRCLASS = {57M25 (57M27)},
  MRNUMBER = {1973052},
MRREVIEWER = {Martin\ Scharlemann},
       DOI = {10.4007/annals.2003.157.433},
       URL = {https://doi.org/10.4007/annals.2003.157.433},
}

@article {DaherPowell,
    AUTHOR = {Daher, Michelle and Powell, Mark},
     TITLE = {Smoothing 3-manifolds in 5-manifolds},
   JOURNAL = {Invent. Math.},
  FJOURNAL = {Inventiones Mathematicae},
    VOLUME = {245},
      YEAR = {2026},
    NUMBER = {1},
     PAGES = {491--511},
      ISSN = {0020-9910,1432-1297},
   MRCLASS = {57K10 (57N35 57N70)},
  MRNUMBER = {5046848},
       DOI = {10.1007/s00222-026-01413-9},
       URL = {https://doi-org.ezproxy.lib.utexas.edu/10.1007/s00222-026-01413-9},
}

@article {Dwyer,
    AUTHOR = {Dwyer, William G.},
     TITLE = {Homology, {M}assey products and maps between groups},
   JOURNAL = {J. Pure Appl. Algebra},
  FJOURNAL = {Journal of Pure and Applied Algebra},
    VOLUME = {6},
      YEAR = {1975},
    NUMBER = {2},
     PAGES = {177--190},
      ISSN = {0022-4049,1873-1376},
   MRCLASS = {55E35 (20J05)},
  MRNUMBER = {385851},
MRREVIEWER = {Timothy\ Porter},
       DOI = {10.1016/0022-4049(75)90006-7},
       URL = {https://doi.org/10.1016/0022-4049(75)90006-7},
}

@article {EilenbergMacLane,
    AUTHOR = {Eilenberg, Samuel and MacLane, Saunders},
     TITLE = {Relations between homology and homotopy groups of spaces},
   JOURNAL = {Ann. of Math. (2)},
  FJOURNAL = {Annals of Mathematics. Second Series},
    VOLUME = {46},
      YEAR = {1945},
     PAGES = {480--509},
      ISSN = {0003-486X},
   MRCLASS = {56.0X},
  MRNUMBER = {13312},
MRREVIEWER = {P.\ A.\ Smith},
       DOI = {10.2307/1969165},
       URL = {https://doi-org.ezproxy.lib.utexas.edu/10.2307/1969165},
}

@book {freedmanquinn,
    AUTHOR = {Freedman, Michael H. and Quinn, Frank},
     TITLE = {Topology of 4-manifolds},
    SERIES = {Princeton Mathematical Series},
    VOLUME = {39},
 PUBLISHER = {Princeton University Press, Princeton, NJ},
      YEAR = {1990},
     PAGES = {viii+259},
      ISBN = {0-691-08577-3},
   MRCLASS = {57N13 (57-02)},
  MRNUMBER = {1201584},
MRREVIEWER = {Ian\ Hambleton},
}

@Article{Heck11,
  author={Heck, Prudence},
  year={2011},
  eprint={https://arxiv.org/abs/1110.6903},
  status={preprint},
  title={Homotopy properties of knots in prime manifolds},
}

@article {Kervaire,
    AUTHOR = {Kervaire, Michel A.},
     TITLE = {Les n\oe uds de dimensions sup\'erieures},
   JOURNAL = {Bull. Soc. Math. France},
  FJOURNAL = {Bulletin de la Soci\'et\'e{} Math\'ematique de France},
    VOLUME = {93},
      YEAR = {1965},
     PAGES = {225--271},
      ISSN = {0037-9484},
   MRCLASS = {57.20},
  MRNUMBER = {189052},
MRREVIEWER = {Edward\ M.\ Brown},
       URL = {http://www.numdam.org/item?id=BSMF_1965__93__225_0},
}

@article {klugmiller_1,
    AUTHOR = {Klug, Michael R. and Miller, Maggie},
     TITLE = {Concordance of surfaces in 4-manifolds and the
              {F}reedman-{Q}uinn invariant},
   JOURNAL = {J. Topol.},
  FJOURNAL = {Journal of Topology},
    VOLUME = {14},
      YEAR = {2021},
    NUMBER = {2},
     PAGES = {560--586},
      ISSN = {1753-8416,1753-8424},
   MRCLASS = {57K45 (57K40)},
  MRNUMBER = {4286049},
MRREVIEWER = {Uwe\ Kaiser},
       DOI = {10.1112/topo.12191},
       URL = {https://doi.org/10.1112/topo.12191},
}

@misc{klugmiller_2,
      title={Concordance of spheres in 4-manifolds with an immersed dual sphere}, 
      author={Michael Klug and Maggie Miller},
      year={2022},
      eprint={2211.07177},
      archivePrefix={arXiv},
      primaryClass={math.GT},
      url={https://arxiv.org/abs/2211.07177}, 
}

@article {Kuzbary,
    AUTHOR = {Kuzbary, Miriam},
     TITLE = {An analogue of {M}ilnor's invariants for knots in 3-manifolds},
   JOURNAL = {Algebr. Geom. Topol.},
  FJOURNAL = {Algebraic \& Geometric Topology},
    VOLUME = {24},
      YEAR = {2024},
    NUMBER = {6},
     PAGES = {3043--3067},
      ISSN = {1472-2747,1472-2739},
   MRCLASS = {57K10 (57K12)},
  MRNUMBER = {4812214},
       DOI = {10.2140/agt.2024.24.3043},
       URL = {https://doi.org/10.2140/agt.2024.24.3043},
}

@article {Levine,
    AUTHOR = {Levine, J.},
     TITLE = {Invariants of knot cobordism},
   JOURNAL = {Invent. Math.},
  FJOURNAL = {Inventiones Mathematicae},
    VOLUME = {8},
      YEAR = {1969},
     PAGES = {98--110; addendum, ibid. 8 (1969), 355},
      ISSN = {0020-9910,1432-1297},
   MRCLASS = {57.10},
  MRNUMBER = {253348},
MRREVIEWER = {R.\ H.\ Fox},
       DOI = {10.1007/BF01404613},
       URL = {https://doi.org/10.1007/BF01404613},
}

@book {MKS,
    AUTHOR = {Magnus, Wilhelm and Karrass, Abraham and Solitar, Donald},
     TITLE = {Combinatorial group theory},
   EDITION = {second},
      NOTE = {Presentations of groups in terms of generators and relations},
 PUBLISHER = {Dover Publications, Inc., Mineola, NY},
      YEAR = {2004},
     PAGES = {xii+444},
      ISBN = {0-486-43830-9},
   MRCLASS = {20E05 (05A15 20E06 20F12)},
  MRNUMBER = {2109550},
}

@article {MillerD,
    AUTHOR = {Miller, David},
     TITLE = {An extension of {M}ilnor's {$\overline \mu$}-invariants},
   JOURNAL = {Topology Appl.},
  FJOURNAL = {Topology and its Applications},
    VOLUME = {65},
      YEAR = {1995},
    NUMBER = {1},
     PAGES = {69--82},
      ISSN = {0166-8641,1879-3207},
   MRCLASS = {57M25 (57M10)},
  MRNUMBER = {1354382},
MRREVIEWER = {Kimihiko\ Motegi},
       DOI = {10.1016/0166-8641(94)00098-N},
       URL = {https://doi.org/10.1016/0166-8641(94)00098-N},
}

@incollection {Milnor,
    AUTHOR = {Milnor, John},
     TITLE = {Isotopy of links},
 BOOKTITLE = {Algebraic geometry and topology. {A} symposium in honor of
              {S}. {L}efschetz},
     PAGES = {280--306},
 PUBLISHER = {Princeton Univ. Press, Princeton, NJ},
      YEAR = {1957},
   MRCLASS = {55.0X},
  MRNUMBER = {92150},
MRREVIEWER = {W.\ S.\ Massey},
}

@article {schneidermanteichner,
    AUTHOR = {Schneiderman, Rob and Teichner, Peter},
     TITLE = {Homotopy versus isotopy: spheres with duals in 4-manifolds},
   JOURNAL = {Duke Math. J.},
  FJOURNAL = {Duke Mathematical Journal},
    VOLUME = {171},
      YEAR = {2022},
    NUMBER = {2},
     PAGES = {273--325},
      ISSN = {0012-7094,1547-7398},
   MRCLASS = {57R52},
  MRNUMBER = {4375617},
MRREVIEWER = {J.\ P. E. Hodgson},
       DOI = {10.1215/00127094-2021-0016},
       URL = {https://doi.org/10.1215/00127094-2021-0016},
}

@article {schwartz,
    AUTHOR = {Schwartz, Hannah R.},
     TITLE = {Equivalent non-isotopic spheres in 4-manifolds},
   JOURNAL = {J. Topol.},
  FJOURNAL = {Journal of Topology},
    VOLUME = {12},
      YEAR = {2019},
    NUMBER = {4},
     PAGES = {1396--1412},
      ISSN = {1753-8416,1753-8424},
   MRCLASS = {57N70 (57K40)},
  MRNUMBER = {4138734},
MRREVIEWER = {Richard\ Stong},
       DOI = {10.1112/topo.12121},
       URL = {https://doi.org/10.1112/topo.12121},
}

@article {Stallings,
    AUTHOR = {Stallings, John},
     TITLE = {Homology and central series of groups},
   JOURNAL = {J. Algebra},
  FJOURNAL = {Journal of Algebra},
    VOLUME = {2},
      YEAR = {1965},
     PAGES = {170--181},
      ISSN = {0021-8693},
   MRCLASS = {18.20},
  MRNUMBER = {175956},
MRREVIEWER = {Hirosi\ Nagao},
       DOI = {10.1016/0021-8693(65)90017-7},
       URL = {https://doi.org/10.1016/0021-8693(65)90017-7},
}

@article {Stees26,
    AUTHOR = {Stees, Ryan},
     TITLE = {Milnor's invariants for knots and links in closed orientable
              3-manifolds},
   JOURNAL = {Geom. Topol.},
  FJOURNAL = {Geometry \& Topology},
    VOLUME = {30},
      YEAR = {2026},
    NUMBER = {3},
     PAGES = {983--1050},
      ISSN = {1465-3060,1364-0380},
   MRCLASS = {57N65},
  MRNUMBER = {5065834},
       DOI = {10.2140/gt.2026.30.983},
       URL = {https://doi.org/10.2140/gt.2026.30.983},
}

@Article{Stees25a,
  author={Stees, Ryan},
  Year = {2025},
  eprint = {https://arxiv.org/abs/2508.14638},
  status={preprint},
  title={Almost-concordance of knots in aspherical 3-manifolds},
}

@article {stong,
    AUTHOR = {Stong, Richard},
     TITLE = {Uniqueness of {$\pi_1$}-negligible embeddings in
              {$4$}-manifolds: a correction to {T}heorem 10.5 of {\it
              {T}opology of 4-manifolds} [{P}rinceton {U}niv. {P}ress,
              {P}rinceton, {NJ}, 1990; {MR}1201584 (94b:57021)] by {M}. {H}.
              {F}reedman and {F}. {Q}uinn},
   JOURNAL = {Topology},
  FJOURNAL = {Topology. An International Journal of Mathematics},
    VOLUME = {32},
      YEAR = {1993},
    NUMBER = {4},
     PAGES = {677--699},
      ISSN = {0040-9383},
   MRCLASS = {57N15},
  MRNUMBER = {1241868},
MRREVIEWER = {Ian\ Hambleton},
       DOI = {10.1016/0040-9383(93)90046-X},
       URL = {https://doi.org/10.1016/0040-9383(93)90046-X},
}

@article {sunukjian,
    AUTHOR = {Sunukjian, Nathan S.},
     TITLE = {Surfaces in 4-manifolds: concordance, isotopy, and surgery},
   JOURNAL = {Int. Math. Res. Not. IMRN},
  FJOURNAL = {International Mathematics Research Notices. IMRN},
      YEAR = {2015},
    NUMBER = {17},
     PAGES = {7950--7978},
      ISSN = {1073-7928,1687-0247},
   MRCLASS = {57R52 (57R65)},
  MRNUMBER = {3404006},
MRREVIEWER = {J.\ P. E. Hodgson},
       DOI = {10.1093/imrn/rnu187},
       URL = {https://doi.org/10.1093/imrn/rnu187},
}

\end{document}